\documentclass{amsart}
\usepackage{amsmath}
\usepackage{amsfonts}
\usepackage{amssymb}
\usepackage{fullpage}

\usepackage{latexsym}
\usepackage{color}
\usepackage{tikz}
\usepackage{lineno}
\usepackage{comment}
\usepackage{soul}

\newtheorem{thm}{Theorem}
\newtheorem{lem}[thm]{Lemma}
\newtheorem{prop}[thm]{Proposition}
\newtheorem{cor}[thm]{Corollary}
\newtheorem{conj}[thm]{Conjecture}
\newtheorem*{thmA}{Theorem A}

\theoremstyle{definition}

\newtheorem{definition}[thm]{Definition}

\providecommand{\ben}{\begin{enumerate}}
\providecommand{\een}{\end{enumerate}}
\providecommand{\bit}{\begin{itemize}}
\providecommand{\eit}{\end{itemize}}
\providecommand{\bc}[2]{{#1\choose#2}}
\providecommand\gbin[2]{\genfrac{[}{]}{0pt}{}{#1}{#2}}
\providecommand{\floor}[1]{\lfloor #1 \rfloor}

\providecommand{\size}[1]{\left\lvert {#1} \right\rvert}
\providecommand{\gen}[1]{\langle #1 \rangle}
\providecommand{\zo}{\mathbf{0}}

\providecommand{\Z}{\mathbb{Z}}
\providecommand{\F}{\mathbb{F}}
\providecommand{\EE}{\mathcal{E}}
\providecommand{\FF}{\mathcal{F}}
\providecommand{\GG}{\mathcal{G}}
\providecommand{\II}{\mathcal{I}}
\providecommand{\LL}{\mathcal{L}}
\providecommand{\Ss}{\mathcal{S}}
\usepackage{hyperref}
\hypersetup{colorlinks=True}

\begin{document}

\title{$3$-cluster-free families of subspaces}

\author[Currier]{Gabriel Currier}
\address{Department of Mathematics \\
University of British Columbia \\
Vancouver, BC, Canada V6T 1Z2}
\email{currierg@math.ubc.ca}

\author[Shahriari]{Shahriar Shahriari}
\address{Department of Mathematics \& Statistics \\
Pomona College \\
Claremont, CA 91711}
\email{sshahriari@pomona.edu}

\subjclass[2000]{}


\date{September 8, 2023}

\begin{abstract}
Three $k$-dimensional subspaces $A$, $B$, and $C$ of an $n$-dimensional vector space $V$ over a finite field are called a \emph{$3$-cluster} if $A \cap B \cap C = \{\mathbf{0}_V\}$ and yet $\dim(A+B+C) \leq 2k$. A special kind of $3$-cluster, which we call a \emph{covering triple}, consists of subspaces $A,B,C$ such that $A = (A \cap B )\oplus (A \cap C)$. We prove that, for $2 \leq k \le n/2$, the largest size of a covering triple-free family of $k$-dimensional subspaces  is the same as the size of the largest such star (a family of subspaces all containing a designated non-zero vector). Moreover, we show that if $k < n/2$, then stars are the only families achieving this largest size. This in turn implies the same result for $3$-clusters, which gives the vector space-analogue of a theorem of Mubayi for set systems.
\end{abstract}
\maketitle

\section{Introduction}

Let $n$ be a positive integer, $q$ a power of prime, $\F_q$ a field of order $q$, and $V$ a vector space of dimension $n$ over $\F_q$. A family of subspaces of $V$ is called \emph{intersecting} if the intersection of each pair of subspaces in the collection is more than just the zero vector. If a family of subspaces satisfies the stronger property that they all contain a specific one-dimensional subspace, then the family is called a \emph{star}. The celebrated Erd\H{o}s-Ko-Rado (EKR) theorem for vector spaces (originally proved for most cases by Hsieh \cite{Hsieh:75}, but more conceptual proofs for all cases in more general settings appeared later) states that there are no intersecting families of $k$-dimensional spaces larger than stars.

\begin{thm}[Hsieh 1975 \cite{Hsieh:75}, Frankl-Wilson 1986 \cite{FranklWilson:86}, Godsil-Newman 2006 \cite{GodsilNewman:06}, Chowdhury-Patk\' os 2010 \cite{ChowdhuryPatkos:10}]\label{thm:vsEKR} Fix integers $k$ and $n$ with $n \geq 2k$, and suppose $\FF$ is an intersecting family of $k$-dimensional subspaces of an $n$-dimensional vector space over a finite field. Then $\size{\FF}$ is no more than the size of the largest such star. Moreover, for $n > 2k$, the only such families with  maximum size are stars.
\end{thm}

A set of two $k$-dimensional subspaces that intersect trivially is also called a \emph{$2$-cluster}, and, so, an intersecting family is the same as a \emph{$2$-cluster-free} family. Now, if $\{A, B\}$ is a $2$-cluster, then $\dim(A+B) = 2k$. By analogy, for $1 \leq k \leq n/2$, three $k$-dimensional subspaces $A$, $B$, and $C$ are called a \emph{$3$-cluster} if $A \cap B \cap C = \{\zo_V\}$ and yet $\dim(A+B+C) \leq 2k$. A family of $k$-dimensional subspaces of $V$ is called \emph{$3$-cluster-free} if it contains no $3$-clusters. A star is, of course, automatically $3$-cluster-free.

Our main goal in this paper is to prove that, for the right parameters, the maximum-sized $3$-cluster-free families are stars. In fact, we will prove a stronger statement. If $A$, $B$, $C$ are $k$-dimensional subspaces of $V$ with $A = (A \cap B) \oplus (A \cap C)$, we say that $A, B, C$ are a \emph{covering triple}. Clearly all covering triples are $3$-clusters as well. We prove

\begin{thmA}[Theorem \ref{thm:Main} below]
Let $V$ be an $n$-dimensional vector space over $\F_q$, and $2 \leq k \leq n/2$. The largest size of a covering-triple-free family is the same as the size of the largest star (of $k$-dimensional subspaces). Moreover, for $2 \leq k < n/2$, stars are the only families achieving this largest size.
\end{thmA}

The original Erd\H{o}s-Ko-Rado theorem  is about families of subsets of $[n] = \{1, 2, \ldots, n\}$. A family of subsets is \emph{intersecting} if every pair of subsets in the collection has a non-empty intersection. If every subset in the collection contains some specific $i \in [n]$, then the family is a \emph{star}.  Erd\H{o}s, Ko, and Rado \cite{ErdosKoRado:61} proved that, for $1 \leq k \leq n/2$, the largest size of an intersecting family of subsets of size $k$ of a set with $n$ elements is $\bc{n-1}{k-1}$. This is, of course, also the size of a largest star of $k$-sets. They also characterized the families of extremal size: if the size of the family is actually equal to $\bc{n-1}{k-1}$ and $n > 2k$, then the family must be a star. This theorem has numerous proofs, many generalizations (see Frankl and Tokushige \cite{FranklTokushige:16} for an excellent survey), and a plethora of fruitful research areas take it as their starting point (see Godsil and Meagher \cite{GodsilMeagher:16}, as an example).

Naively, one may be tempted to see what happens if---instead of forbidding a pair of subsets $A$, $B$ with $A \cap B = \emptyset$---we forbid, in our family, the existence of three subsets $A$, $B$, and $C$ with $A \cap B \cap C = \emptyset$. If the size of the family is at least three, then any such family will also be intersecting and so, as long as $n \geq 2k$, the original EKR theorem already answers the question. (Frankl \cite{Frankl:76} showed that, for this three set version, the conclusion remains valid even if $n \geq 3k/2$.) Hence, to get an interesting question with three sets, we need to forbid fewer configurations, and so in addition to $A \cap B \cap C = \emptyset$, we need additional restrictions on the forbidden configuration. Erd\H{o}s \cite{Erdos:71}, Chv\'atal \cite{Chvatal:74}, Mubayi and Verstra\"ete \cite{MubaiVerstraete:05} and Keevash and Mubayi \cite{KeevashMubayi:10} have worked on and proposed such questions.

More directly relevant to our project,  Frankl and F\"uredi \cite{FranklFuredi:83}---partially answering a question of Katona---showed, somewhat surprisingly, that, for $n > k^2+3k$, the maximum size of a family of $k$-subsets of $[n]$ that contains no $A$, $B$, and $C$ with $A \cap B \cap C = \emptyset$ and $\size{A \cup B \cup C} \le 2k$ continues to be $\displaystyle{\bc{n-1}{k-1}}$. This result was extended $n \geq 3k/2$ (for $k \geq 3$) by Mubayi \cite{Mubayi:06}, who then asked a similar question for so-called $d$-clusters, although he did not use that terminology at the time. For $d \geq 2$,  a \emph{$d$-cluster} is defined to be a collection of $d$ subsets of size $k$ of $[n]$ with trivial intersection whose union has at most $2k$ elements. Having proved that the size of a $3$-cluster-free family can be no more than the size of a star, Mubayi conjectured that, in fact, the size of any $d$-cluster-free family of $k$-subspaces, for $n \geq dk/(d-1)$, and $3 \leq d \leq k$, can be no more than the size of a star. For large $n$, the conjecture was settled by Mubayi \cite{Mubayi:07}, Mubayi and Ramadurai \cite{MubayiRam:09} and independently F\"{u}redi and \"{O}zkahya \cite{FurediOzkahya:11}. In Currier \cite{Currier:21}, the first author settled the conjecture in the affirmative for all $n \geq dk/(d-1)$.

Investigating vector space analogues of problems for set systems has a long history going back to at least Gian Carlo Rota.  Notable successes, other than the EKR theorem for vector spaces discussed earlier, are the proof of  the vector space analogues of Lov\'{a}sz's version of the Kruskal-Katona theorem (Chowdhury and Patk\' os \cite{ChowdhuryPatkos:10}), the Hilton-Milner theorem (Blokuis, Brouwer, Chowdhury, Frankl, Mussche, Patk\' os, and Sz\H onyi\cite{Blokhuisetal:10}), and the proof of Manickam--Mikl\'{o}s--Singhi conjecture (Chowdhury, Sarkis, and Shahriari \cite{ChoSarSha:14}, Ihringer \cite{Ihringer:16}, and Huang and Sudakov \cite{HuangSudakov:14}). More recently, and related to the EKR point of view, there has been progress in finding or bounding the maximum size of families of subspaces (not all necessary of the same dimension) avoiding a forbidden configuration. Sarkis, Shahriari, and PCURC \cite{PCURC-S11:14} found an upper bound for diamond-free families of subspaces, Shahriari and Yu \cite{ShahYu:20} and Xiao and Tompkins \cite{XiaoTompkins:20} found the maximum size of families avoiding certain brooms, forks, and butterflies. 

While there are cases that a problem posed for the poset of subspaces is easier than the one for sets (e.g., the existence of a partition of the poset into a minimum number of chains with as equal length as possible---the so-called F\"{u}redi partition---has been completely settled for vector spaces by Hsu, Logan, and Shahriari \cite{HsuLogSha:06}, while for sets it has only been recently settled for very large $n$ by Sudakov, Tomon, and Wagner \cite{SudakovTomonWagner:22}) most of the time, an extremal problem poses new challenges for vector spaces since many of the central techniques in extremal set theory do not directly translate to vector spaces. As a result, many open problems remain.

The corollary to our main result---that under the obvious conditions, the maximum size of a covering triple-free collection is bounded by the size of a star---gives immediately a tight bound on the maximum size of a $3$-cluster-free family. This gives a direct vector space analogue of Mubayi's aforementioned result \cite{Mubayi:06} for sets. In addition, our techniques are purely combinatorial, differ from Mubayi's, and are, as far as we know, new. 

In the next section we will introduce the necessary preliminaries and background information surrounding vectors spaces over finite fields. In section \ref{sec:sketch}, we will sketch an idea for the proof that we hope will aid in understanding the main results. In section \ref{sec:main}, we prove Theorem A.

\section{Notation and Preliminaries}

We denote by $\LL(V)$ the poset of subspaces of $V$ ordered by inclusion, and, for $0 \leq k \leq n$, $\gbin{V}{k}$ will denote the collection of subspaces of dimension $k$ in $\LL(V)$. Let $U, W \in \LL(V)$. If $U \cap W = \{\zo_V\}$, then the two subspaces are said to be \emph{skew}. If two subspaces are not skew, they are said to \emph{intersect non-trivially}. For $U \in \LL(V)$, by abuse of notation, $\displaystyle{\gbin{V - U}{k}}$ will denote all $k$ dimensional subspaces of $V$ that are skew to $U$.

In analogy with the binomial coefficients, define $[n]_q = \dfrac{q^n-1}{q-1} = q^{n-1} + \cdots + q + 1$, $[n]_q! = [n]_q [n-1]_q \cdots [2]_q [1]_q$, and $\displaystyle{\gbin{n}{k}_q = \frac{[n]_q!}{[k]_q!\ [n-k]_q!}}$ for $0 \leq k \leq n$. The integers $\displaystyle{\gbin{n}{k}_q}$ are called \emph{Gaussian (or $q$-binomial) coefficients}. The next lemma gives a few  well-known facts about these integers. 

\begin{lem}\label{lem:prelimsub}
Let $n$ be a positive integer, $q$ a prime power, and $V$ a vector space of dimension $n$ over $\F_q$. 
\ben
\renewcommand{\theenumi}{\alph{enumi}}
\item\ \label{lem:subspacecount} Let $0 \leq m, i \leq n$, and $\displaystyle{W \in \gbin{V}{m}}$ be fixed. Then 
$$\displaystyle{\size{\gbin{V - W}{i}} = q^{mi} \gbin{n-m}{i}_q}.$$
In particular, $\displaystyle{\size{\gbin{V}{i}} = \gbin{n}{i}_q}$, and  the number of subspaces $U$ with $V = W \oplus U$ is $q^{m(n-m)}$.
\item\ \label{lem:qrecurrence} For  $k \leq n$, we have
$$\gbin{n}{n-k}_q = \gbin{n}{k}_q  = q^{n-k}\gbin{n-1}{k-1}_q + \gbin{n-1}{k}_q.$$
\renewcommand{\theenumi}{\arab{enumi}}
\een
\end{lem}

We also need a result of Gerbner and Patk\'os \cite[Corollary p. 2867]{GerbnerPatkos:09}

\begin{thm}[Gerbner and Patk\'os \cite{GerbnerPatkos:09}]\label{thm:GerbnerPatkos}
Let $q$ be a prime power, $V$ a vector space of dimension $n$ over $\F_q$, $k < n/2$, and $n/2 < \ell \leq n-k$. Let $\GG_k \subseteq \gbin{V}{k}$, $\GG_{\ell} \subseteq \gbin{V}{\ell}$, and assume $\GG_k \cup \GG_{\ell}$ is intersecting. Then,
\begin{equation}
\frac{q^\ell\gbin{n-k}{\ell}_q}{\gbin{n-\ell-1}{k-1}_q} \size{\GG_k} + \size{\GG_{\ell}} \leq \gbin{n}{\ell}_q.\label{eq:GerbnerPatkos}
\end{equation}
\end{thm}

For our purposes, we need a slight strengthening of a special case of Theorem \ref{thm:GerbnerPatkos}.

\begin{cor}\label{cor:EKRtypenonuniform} 
Let $q$ be a prime power, $V$ a vector space of dimension $n$ over $\F_q$, and $1 \leq k \le n/2$. Let $\GG_k \subseteq \gbin{V}{k}$, $\GG_{n-k} \subseteq \gbin{V}{n-k}$, and assume $\GG_k \cup \GG_{n-k}$ is intersecting. Let $\alpha \geq \begin{cases} 1 & \textrm{if $k = n/2$} \\ q^{n-k} & \textrm{if $k < n/2$.} \end{cases}$ Then,
\begin{equation}
\alpha \size{\GG_k} + \size{\GG_{n-k}} \leq \alpha \gbin{n-1}{k-1}_q + \gbin{n-1}{k}_q.\label{eq:GkGnkinequality}
\end{equation}
Furthermore, if Equation \eqref{eq:GkGnkinequality} is an equality, then $\size{\GG_{n-k}} \geq \gbin{n-1}{k}_q$.
\end{cor}

\begin{proof}
In the case $k= n/2$, both $\GG_k$ and $\GG_{n-k}$ are intersecting families in $\gbin{V}{n/2}$ and $\gbin{n-1}{k-1}_q = \gbin{n-1}{k}_q = \gbin{n-1}{n/2}_q$. Hence, by the regular EKR theorem for vector spaces, Theorem \ref{thm:vsEKR}, each of $\size{\GG_k}$ and $\size{\GG_{n-k}}$ are no more than $\gbin{n-1}{n/2}_q$, and Equation \eqref{eq:GkGnkinequality} follows. Also, to have equality, we need $\size{\GG_{n-k}} = \gbin{n-1}{k}_q$.  Hence, we can assume  $k < n/2$ and $\alpha \ge q^{n-k}$. Letting $\ell = n-k$ in Equation \eqref{eq:GerbnerPatkos} of Theorem \ref{thm:GerbnerPatkos}, and using  Lemma \ref{lem:prelimsub}\eqref{lem:qrecurrence}, we have
\begin{equation}
q^{n-k}\size{\GG_k} + \size{\GG_{n-k}} \leq \gbin{n}{n-k}_q = \gbin{n}{k}_q = q^{n-k} \gbin{n-1}{k-1}_q + \gbin{n-1}{k}_q.
\end{equation}
 By Theorem \ref{thm:vsEKR}, $\displaystyle{\size{\GG_k} \leq \gbin{n-1}{k-1}_q}$.  Thus, since  $\alpha \geq q^{n-k}$, we have 
\begin{align*} \alpha \size{\GG_k} + \size{\GG_{n-k}}  & = q^{n-k} \size{\GG_k} + \size{\GG_{n-k}} + (\alpha - q^{n-k})\size{\GG_k} \\ 
	& \leq q^{n-k} \gbin{n-1}{k-1}_q + \gbin{n-1}{k}_q + (\alpha - q^{n-k})\gbin{n-1}{k-1}_q  \\ 
	&= \alpha \gbin{n-1}{k-1}_q + \gbin{n-1}{k}_q,
\end{align*}
as predicted by Equation \eqref{eq:GkGnkinequality}. Since $\size{\GG_k} \leq \gbin{n-1}{k-1}_q$,  to have equality in Equation \eqref{eq:GkGnkinequality}, we  need $\size{\GG_{n-k}} \geq \gbin{n-1}{k}_q$. This completes the proof. \qedhere
\end{proof}

\section{Idea of the Proof for the Main Theorem}\label{sec:sketch}

To prove that the largest covering-triple-free family is as big as a star, we aspire to mimic a particular way of counting the number of $k$-dimensional subspaces in a star.

\begin{lem}\label{lem:starcount}
Let $q$ be a prime power, $V$ a vector space of dimension $n$ over $\F_q$, $\zo_V \neq v \in V$, and $1 \leq k \leq n/2$. Let $\FF$ be the family of all $k$-dimensional subspaces of $V$ that contain $v$, and let $A \in \FF$ be fixed. Finally, for $1 \leq i \leq k$,  let $\FF_i$ consist of those subspaces in $\FF$ whose intersection with $A$ has dimension exactly $i$. Then
$$\size{\FF_i} = q^{(k-i)^2} \gbin{k-1}{i-1}_q \gbin{n-k}{k-i}_q.$$
As a result, 
$$\gbin{n-1}{k-1}_q = \size{\FF} = \sum_{i=1}^{k} q^{(k-i)^2} \gbin{k-1}{i-1}_q \gbin{n-k}{k-i}_q.$$
\end{lem}

\begin{proof}
If $B \in \FF_i$, then $\dim(A \cap B) = i$, and $\dim(A+B) = 2k-i$ (see Figure \ref{fig:starcounting}). There are as many subspaces of dimension $2k-i$ that contain $A$, as there are subspaces of dimension $k-i$ in $V/A$, a vector space of dimension $n-k$. Hence, the number of choices for $A+B$ is $\displaystyle{\gbin{n-k}{k-i}_q}$. Likewise, the number of choices for $A \cap B$, is the same as the number of subspaces of dimension $i-1$ of $A/\gen{v}$, and is equal to $\displaystyle{\gbin{k-1}{i-1}_q}$. Having fixed $A+B$ and $A \cap B$, the number of choices for $B$ is the same as the number of $(k-i)$- dimensional subspaces of $(A+B)/(A \cap B)$ that are skew to $A/(A \cap B)$ (see Figure \ref{fig:starcounting}), which , by Lemma \ref{lem:prelimsub}\eqref{lem:subspacecount}, is equal to $q^{(k-i)^2}$. Hence, $\size{\FF_i} = q^{(k-i)^2} \gbin{k-1}{i-1}_q \gbin{n-k}{k-i}_q$. The rest follows since $\FF$ is a disjoint union of $\FF_i$, and the number of $k$-dimensional subspaces of $V$ that contain $\gen{v}$ is $\gbin{n-1}{k-1}_q$. \qedhere
\end{proof}

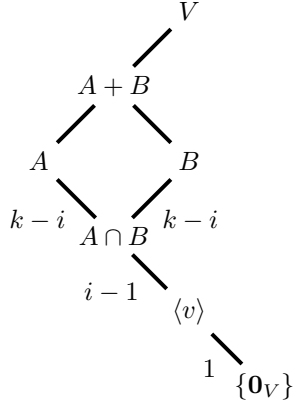
\begin{figure}[htbp]
\begin{tikzpicture}
\node (0) at (1,-1) {$\{\zo_V\}$};
\node (v) at (0,0){$\gen{v}$};
\node (A) at (-2,2) {$A$};
\node (AiB) at (-1,1) {$A \cap B$};
\node (B) at (0,2) {$B$};
\node (AB) at (-1,3) {$A+B$};
\node (V)at (0,4) {$V$};
\draw[line width=1.5pt] (0)--node[below left]{$1$}(v)--node[below left]{$i-1$}(AiB)--node[below left]{$k-i$}(A)--(AB)--(V) (AB)--(B)--node[below right]{$k-i$}(AiB);
\end{tikzpicture}
\caption{$A$ and $B$ are $k$-dimensional subspaces of the $n$-dimensional space $V$. $A \cap B$ is an $i$-dimensional subspace containing $\gen{v}$, and $(A+B)/A$ is a $k-i$ dimensional subspace of the $(n-k)$-dimensional space $V/A$.\label{fig:starcounting}}
\end{figure}

Now, suppose we want to apply the same count to an arbitrary covering-triple-free family $\FF$ of $k$-dimensional subspaces. We will again fix $A \in \FF$, and will want to divide the rest of $\FF$ into subfamilies $\FF_i$. If $B \in \FF$ is such that $\dim(A\cap B) = i \neq 0$, then we put $B$ into $\FF_i$. If $A \cap B = \{\zo_V\}$, then $B$ will be assigned to $\FF_i$ if $i$ is the minimum non-zero value of $\dim(B \cap C)$ among all $C \in \FF$. But what if $B$ is skew to \emph{every} element of $\FF$? Such subspaces---which we will collect in a family called $\FF^\star$--- will need to be dealt with separately, and this results in annoying extra cases. However, at the end, there can't be too many such subspaces in a covering-triple-free family, and they are harmless. We choose to include them in every $\FF_i$.

After these modifications, $\size{\FF}$ will continue to be bounded by $\sum_{i=1}^k \size{\FF_i}$, and so we need effective bounds for each of $\size{\FF_i}$. To give a somewhat vague description of our method for bounding $\size{\FF_i}$, ignore the complications, and assume that we have fixed $A \in \FF$ in such a way that every other element of $\FF$ intersects $A$ non-trivially. In such a happy situation, $\FF_i$ will consist merely of subspaces in $\FF$ whose intersection with $A$ is $i$-dimensional. 

We say---for the sake of this heuristic argument---that $B \in \FF_i$ \emph{claims} a subspace $\displaystyle{D \in \gbin{V - A}{k-i}}$ if $D \oplus (B \cap A) = B$. Now, it is certainly possible for a number of $B$'s to claim the \emph{same} $D$. However,  if $B_1$, $\ldots$, $B_r$ all claim $D$, then all of these $B$'s are in the same $\FF_i$, but more importantly, $B_1 \cap A$, $\ldots$, $B_r \cap A$ form an intersecting family inside $A$ (the assumption that $\FF$ is covering-triple-free  is essential here). As a result, Theorem \ref{thm:vsEKR}  will give us a non-trivial bound on the number of $B$'s that claim the same $D$, which in turn will give a bound on $\size{\FF_i}$ . Since we are using the EKR theorem for subspaces of a $k$-dimensional space $A$, the argument, unfortunately, will work only in the cases when $1 \leq i \leq k/2$.

When $i > k/2$, it is not possible  to establish an upper bound for each $\size{\FF_i}$ individually. For the final twist of the proof, we modify our methods, and instead  bound $\size{\FF_i} + \size{\FF_{k-i}}$, for each $1 \leq i < k/2$ (if $k$ is even, we continue to also use the earlier bound for $\size{\FF_{k/2}}$). The important point is that our upper bounds are precise enough that in every instance they mirror what would have happened if $\FF$ was a star. Hence, at the end, we can conclude that $\size{\FF}$ is no more than the size of a star.

In the next section, this plan will be carried out.

\section{Proof of Main Theorem}\label{sec:main}
For the rest of the paper, as before, $n$ is a positive integer, $q$ a power of prime, $\F_q$ a field of order $q$, and $V$ a vector space of dimension $n$ over $\F_q$. Furthermore $2 \le k \leq n/2$ is an integer.

The next (long) definition provides the basic infrastructure/notation for our proof.

\begin{definition}\label{def:iIphi}
Let $\FF \subseteq \gbin{V}{k}$. Define $\FF^\star$ to denote those elements of $\FF$ that are skew to every other element of $\FF$. Fix $A \in \FF$, and define a function $i_A \colon \FF \setminus \FF^* \to \{1, \ldots, k\} \subset \Z^{\geq 0}$ by
$$i_A(B) = \begin{cases}  \dim{(A \cap B)} & \textrm{if}\ A \cap B \neq \{\zo_V\} \\ \min\{ \dim{(B \cap C)} \mid C \in \FF, B \cap C \neq \{\zo_V\}\} & \textrm{if}\ A \cap B = \{\zo_V\}\end{cases}$$
Also, define $\II_A \colon \FF \setminus \FF^* \to \mathbf{2}^\FF$ by
\begin{align*}
\II_A(B) & = \begin{cases} 
\{A\} & \textrm{if}\ A \cap B \neq \{\zo_V\} \\
\{C \in \FF\ \mid \dim{(B \cap C)} = i_A(B)\} & \textrm{if}\ A \cap B = \{\zo_V\} \end{cases}
\end{align*}
In addition, define $\phi_A \colon \FF \to \mathbf{2}^{\LL(V)}$ by
\begin{align*}
\phi_A(B) & =\begin{cases} 
 \{ D \in \gbin{B}{k-i_A(B)} \mid \exists\ C \in \II_A(B)\ \textrm{with}\ D \cap C = \{\zo_V\}\} & \textrm{if}\ B \notin \FF^\star \\  \{D \mid D\ \textrm{a subspace of}\ B\} & \textrm{if}\ B \in \FF^\star. \end{cases}
 \end{align*}
 If $D \in \phi_A(B)$, then we say that $B$ \emph{claims} $D$.
 
 Finally, for $1 \leq i \leq k-1$, define
$$\FF_{i,A} = \{B \in \FF \mid i_A(B) = i\} \cup \FF^\star.$$
To avoid clutter, we will use $\FF_i$ to mean $\FF_{i,A}$.
\end{definition}
 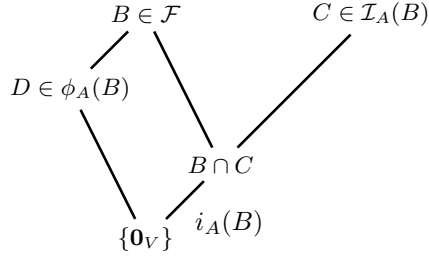
\begin{figure}[htbp]
 \begin{tikzpicture}
\small
\node (z) at (0,0){$\{\zo_V\}$};
\node (b) at (0,3){$B \in \FF$};
\node (Ib) at (3,3){$C \in \II_A(B)$};
\node(bcIb) at (1,1){$B \cap C$};
\node (phib) at (-1,2){$D \in \phi_A(B)$};
\normalsize
\path[line width=1pt] (z) edge (phib) 
(phib) edge (b)
(z) edge node[below right]{$i_A(B)$} (bcIb) 
(bcIb) edge (Ib)
(bcIb) edge (b);
\end{tikzpicture}
\caption{$\dim(B \cap C) = i_A(B)$. $C$ could be $A$ or not, but, regardless, $D \cap A = \{\zo_V\}$ for all $D \in \phi_A(B)$\label{fig:phiAB}}
\end{figure}

The following are straightforward (See Figure \ref{fig:phiAB}) and will be used often:

\begin{lem}\label{lem:prelim}
Let $\FF \subseteq \gbin{V}{k}$, and let $A,B \in \FF$. Then
\ben
\renewcommand{\theenumi}{\alph{enumi}}
\item\ If $B \notin \FF^*$ and $C \in \II_A(B)$, then $\dim(B \cap C) = i_A(B)$.
\item\ \label{lem:prelimc} If $B \not\in \FF^\star$, then $D \in \phi_A(B)$, if and only if there exists $C \in \II_A(B)$ with $C \cap D = \{\zo_V\}$, and $B = D \oplus (B \cap C)$. In particular, if $D \in \phi_A(B)$ and $\II_A(B) = \{A\}$, then $B = D \oplus (B \cap A)$.
\item\ \label{lem:prelimb} If $B \cap A = \{\zo_V\}$ and $B \cap C \neq \{\zo_V\}$ for some $C \in \FF$, then $\dim(B \cap C) \geq i_A(B)$. If, additionally, $D \in \phi(B)$ and $C \cap D = \{\zo_V\}$, then $B = D \oplus (B \cap C)$.
\item\ \label{lem:prelime} For $B\not\in \FF^\star$, and writing $i$ for $i_A(B)$, we have $\displaystyle{\size{\phi_A(B)} \geq q^{i(k-i)}}$, with equality only if  $C \cap B = C^\prime \cap B$ for all $C, C^\prime \in \II_A(B)$.
\item\ \label{lem:prelimf} For $B \in \FF^\star$, and for $1 \leq i < k$, the number of subspaces of dimension $i$ in $\phi_A(B)$ is $\gbin{k}{i} > q^{i(k-i)}$.
\renewcommand{\theenumi}{\arabic{enumi}}
\een
\end{lem}

\begin{proof} (a) is just a restatement of the definitions.
\ben
\renewcommand{\theenumi}{\alph{enumi}}
\addtocounter{enumi}{1}
\item\ Note that $D \cap (B \cap C) \subseteq D \cap C = \{\zo_V\}$, $\dim(D) =k -i_A(B)$, and $\dim (B \cap C) = i_A(B)$. 
\item\ The first part is a restatement of definitions. To see the second, we note $D \subset B$ and $\dim(B \cap C) \ge i_A(B) = k-\dim(D)$.
\item\  By Lemma \ref{lem:prelimsub}\eqref{lem:subspacecount}, each element of $\II_A(B)$ will contribute $q^{i(k-i)}$ subspaces to $\phi_A(B)$. If two elements of $\II_A(B)$ have different intersections with $B$, then these sets of $q^{i(k-i)}$ will be distinct, leading to strict inequality. 
\item\ Note that, by Lemma \ref{lem:prelimsub}\eqref{lem:subspacecount}, $q^{i(k-i)}$ is the number of subspaces of dimension $i$ that have a trivial intersection with one particular subspace of dimension $k-i$, while $\phi_A(B)$ consists of every subspace of dimension $i$. \qedhere
\renewcommand{\theenumi}{\arabic{enumi}}
\een
\end{proof}

For $1 \leq i \leq k/2$, we will establish a bound for $\size{\FF_i}+ \size{\FF_{k-i}}$ (Theorem \ref{thm:FFiFFkmi}). The process will be slightly involved, but will proceed more-or-less by bounding the number of subspaces in $\FF$ that claim a given subspace $D$. Note that if $D \in \phi_A(B)$ for $B \in \FF$, then $D \cap A = \{\zo_V\}$, and so we only consider subspaces $D$ that are skew to $A$. We begin by introducing a  notation for such subspaces.

\begin{definition} Let $\FF \subseteq \gbin{V}{k}$, and fix $A \in \FF$. Let $D$ be a subspace of $V$ with $D \cap A = \{\zo_V\}$. Define $$\phi_A^{-1}(D) = \{B \in \FF \mid D \in \phi_A(B)\}.$$
\end{definition}

Immediate from the definition is that if $B \in \phi^{-1}(D)$, then $D \subset B$. Furthermore, if $D$ is a subspace skew to $A$ and there exists $B \in (\phi^{-1}(D) \cap \FF^*$), then $B$ is the only element of $\FF$ intersecting $D$ non-trivially. In particular, if $E$ is any subspace with $E \cap D \neq \{\zo_V\}$, then $\phi^{-1}_A(E) \subset \{B\}$.
\\\\
The next lemma will show a basic intersection property for spaces in $\phi_A^{-1}(D)$ and $\phi_A^{-1}(E)$, where $D \subset E$.

\begin{lem}\label{lem:intersectsornot}
Let $\FF \subseteq \gbin{V}{k}$ be a covering-triple-free family.  Fix $A \in \FF$. Let $E$ be a subspace of $V$ with $E \cap A = \{\zo_V\}$, and let $D$ be a subspace of $E$. Assume $B_1 \in \phi_A^{-1}(E)$, and $B_2 \in \phi_A^{-1}(D)$, and suppose $B_1,B_2 \notin \FF^*$.
  Then, if $C_2 \in \II_A(B_2)$ is skew to $D$, it follows that $B_1 \cap C_2 \neq \{\zo_V\}$. 

\end{lem}

\begin{proof}

Note first that $D \subseteq E \subseteq B_1$. If $B_1 \cap C_2 = \{\zo_V\}$, then since $B_2 = D \oplus (B_2 \cap C_2)$ (by lemma \ref{lem:prelim}(\ref{lem:prelimc}) if $B$ intersects $A$ non-trivially, and (\ref{lem:prelimb}) otherwise), we have that $B_2,B_1,C_2$ form a covering triple.\qedhere

 \end{proof}

The next result shows, essentially, that all elements of $\phi^{-1}(D)$ share a common form; that is, there exists $C \in \FF$ such that $B = D \oplus (B \cap C)$ for all $B \in \phi^{-1}(D)$. In actuality, we show something slightly different, involving both $D$ and another subspace $E$ containing $D$, since we are trying to bound $\FF_i$ and $\FF_{k-i}$ simultaneously.

\begin{prop}\label{prop:CstarforDE}
Let $\FF \subset \gbin{V}{k}$ be a covering-triple-free family, and suppose $k/2 \le e \le k-1$. Fix $A \in \FF$, let $E \in \gbin{V - A}{e}$, and suppose either $D = E$ or  $D \in \gbin{E}{k-e}$. Suppose $\phi_A^{-1}(E)$ is non-empty but $\phi_A^{-1}(E) \cap \FF^* = \emptyset$. Then there exists $C \in \FF$ such that $B = D \oplus (B \cap C)$ for all $B \in \phi_A^{-1}(D)$.
\end{prop}

\begin{proof}
We note first that since $\phi_A^{-1}(E) \cap \FF^* = \emptyset$, we must have $\phi_A^{-1}(D) \cap \FF^* = \emptyset$, since $D$ is contained in $E$. Let $B^* \in \phi_A^{-1}(E)$, and let $C^* \in \II_A(B^*)$ be such that $C^* \cap E = \emptyset$. Letting $B \in \phi_A^{-1}(D)$, we claim that setting $C = C^*$ will suffice. The proof will be divided into cases, depending on whether or not $B \cap A = \{\zo_V\}$ and whether $D=E$ or $D \in \gbin{E}{k-e}$.
\\\\
{\bfseries Case 1:} Suppose $B \cap A \neq \{\zo_V\}$. Then, $\II_A(B) = \{A\}$, and since $D \subset E$, by Lemma \ref{lem:intersectsornot}, $B^* \cap A \neq \{\zo_V\}$. Thus, $C^* = A$, so by Lemma \ref{lem:prelim}\eqref{lem:prelimc} we have $B = D \oplus (B \cap C^*)$.
\\\\
{\bfseries Case 2(a):} Suppose $B \cap A = \{\zo_V\}$ and $D = E$. By Lemma \ref{lem:intersectsornot}, we know that $B \cap C^* \neq \{\zo_V\}$. Since $D \cap C^* = \{\zo_V\}$, we conclude by Lemma \ref{lem:prelim}\eqref{lem:prelimb} that $B = D \oplus (B \cap C^*)$.
\\\\
{\bfseries Case 2(b):} Suppose $B \cap A = \{\zo_V\}$ and $D \in \gbin{E}{k-e}$. Since $B \cap B^* \neq \{\zo_V\}$, by Lemma \ref{lem:prelim}\eqref{lem:prelimb} we know $\dim(B^* \cap B) \ge i_A(B) = k-\dim(B^* \cap C^*)$. If we had $B \cap C^* = \{\zo_V\}$, then $B^*,C^*,B$ would form a covering triple, so we know $B \cap C^* \neq \{\zo_V\}$. Since $D \cap C^* \subset E \cap C^* = \{\zo_V\}$, we conclude by Lemma \ref{lem:prelim}\eqref{lem:prelimb} that $B = D \oplus (B \cap C^*)$.
\end{proof}

Proposition \ref{prop:CstarforDE} allows us then to prove the following.

\begin{prop}\label{prop:phiinverseDED}
Let $2 \le k \le n/2$, and suppose $\FF \subseteq \gbin{V}{k}$ is a covering-triple-free family.   Fix $A \in \FF$, and let $1 \le d \le k/2$. Suppose $D \in \gbin{V - A}{d}$, and suppose that every element of $\EE_D \subset \gbin{V -A}{k-d}$ contains $D$. If $|\EE_D| \ge \begin{cases} q^{k-d} & \textrm{if}\ d < k/2 \\ 1 & \textrm{if}\ d = k/2 \end{cases}$, then
\begin{equation}
	\size{\phi_A^{-1}(D)} + \sum_{E \in \EE_D} \size{\phi_A^{-1}(E)} \leq \gbin{k-1}{d}_q + \size{\EE_D} \gbin{k-1}{d-1}_q. \label{eq:phiinvDE}
\end{equation}
Moreover, if Inequality \eqref{eq:phiinvDE} is an equality, then $\size{\phi_A^{-1}(D)} \geq \displaystyle{\gbin{k-1}{d}_q}$.\end{prop}

\begin{proof}
Let $E \in \EE_D$. If any elements of $\phi_A^{-1}(D)$ or $\phi_A^{-1}(E)$ are in $\FF^*$, then $\size{\phi_A^{-1},(D)}$ and $\size{\phi_A^{-1}(E)}$ are both at most $1$, and the inequality is trivially satisfied. So assume no elements of $\phi_A^{-1}(D)$ or $\phi_A^{-1}(E)$ (for any $E \in \EE_D$) are in $\FF^\star$. 

From among elements of $\EE_D$, let $E$ be one with $\size{\phi_A^{-1}(E)}$ as large as possible. By Proposition \ref{prop:CstarforDE}, there exists $C^\star \in \FF$ with $C^\star \cap E = \{\zo_V\}$, $B = E \oplus (C^\star \cap B)$ for all $B \in \phi_A^{-1}(E)$ \emph{and} $B^\prime = D \oplus (C^\star \cap B^\prime)$ for all $B^\prime \in \phi_A^{-1}(D)$.

Let $\GG_d = \{ C^\star \cap B \mid B \in \phi_A^{-1}(E)\}$, and, likewise, $\GG_{k-d} = \{ C^\star \cap B^\prime \mid B^\prime \in \phi_A^{-1}(D)\}$. $\GG_d$ consists of $\size{\phi_A^{-1}(E)}$ subspaces of dimension $d$ of $C^\star$. Likewise,  $\GG_{k-d}$ consists of $\size{\phi_A^{-1}(D)}$ subspaces of dimension $k-d$ of $C^\star$. Now, let $B_1 , B_2 \in \phi_A^{-1}(E) \cup \phi_A^{-1}(D)$, and assume without loss of generality that if at least one $B_i$ is in $\phi_A^{-1}(E)$, then $B_1 \in \phi_A^{-1}(E)$. Under this assumption, we always have $B_2 = (B_2 \cap C^*) + (B_2 \cap B_1)$. Since $\FF$ is covering-triple-free, we must have $C^\star \cap B_1 \cap B_2 \neq \{\zo_V\}$. In particular, $C^\star \cap B_1$ intersects $C^\star \cap B_2$ non-trivially. Hence, $\GG_d \cup \GG_{k-d}$ is an intersecting family of subspaces of $C^\star$. 

    Corollary \ref{cor:EKRtypenonuniform} now applies, and we get
    \begin{equation}\size{\EE_D}\size{\GG_d} + \size{\GG_{k-d}} \leq \size{\EE_D} \gbin{k-1}{d-1}_q + \gbin{k-1}{d}_q. \label{eq:Thm4toRescue}
    \end{equation}
    Inequality \eqref{eq:phiinvDE} now follows since $\sum_{E \in \EE_D} \size{\phi_A^{-1}(E)} \leq \size{\EE_D}\size{\GG_d}$ and $\size{\phi_A^{-1}(D)} = \size{\GG_{k-d}}$. If Inequality \eqref{eq:phiinvDE} is an equality, then we have equality in Inequality \eqref{eq:Thm4toRescue}, and by Corollary \ref{cor:EKRtypenonuniform}, we have $\size{\phi_A^{-1}(D)} = \size{\GG_{k-d}} \geq \gbin{k-1}{d}_q$.
\end{proof}

To be able to control $\size{\FF_i} + \size{\FF_{k-i}}$, for $1 \leq i \le k/2$, we need a well-known generalization of the Marriage Theorem. We include the proof for completeness.

\begin{thm}\label{thm:normalizedmatching}
Let $G = (X, \Delta, Y)$ be a bipartite graph with vertex set $X \cup Y$, edge set $\Delta$, and with all edges having one end in $X$ and the other in $Y$. Assume that there is a positive integer $m$ such that, for all $S \subseteq X$, $\size{N_Y(S)} \geq m \size{S}$. ($N_Y(S)$ is the set of neighbors in $Y$ of the vertices in $S$.) Then there exists a one to $m$ matching between the vertices in $X$ and a subset of vertices in $Y$. 
\end{thm} 

\begin{proof}
Let the set $\tilde{X}$ consists of $m$ copies of $X$.  Construct a new bipartite graph $\tilde{G} = (\tilde{X}, \tilde{\Delta}, Y)$ where each clone of $x \in X$ is adjacent to exactly the neighbors of $x$ in the original graph $G$. The new graph $\tilde{G}$ satisfies the Hall Marriage condition and so there is a one to one matching from elements of $X$ to a subset of elements of $Y$. The pull-back of this matching in $G$ is a one to $m$ matching from $X$ to a subset of vertices in $Y$.
\end{proof}

We will now use Theorem \ref{thm:normalizedmatching} to control $\size{\FF_i} + \size{\FF_{k-i}}$ when $1 \le i \le k/2$. 
\begin{thm}\label{thm:FFiFFkmi}
Let $2 \leq k \leq n/2$, $\FF \subseteq \gbin{V}{k}$ a covering-triple-free family, and fix $A \in \FF$. Then, for $1 \leq i \leq k/2$,
\begin{equation}
\size{\FF_i} + \size{\FF_{k-i}} \leq q^{(k-i)^2}\gbin{n-k}{k-i}_q \gbin{k-1}{i-1}_q + q^{i^2} \gbin{n-k}{i}_q \gbin{k-1}{i}_q.\label{eq:sizeFiFk-i}
\end{equation}
If Inequality \eqref{eq:sizeFiFk-i} is an equality, then $\size{\phi_A^{-1}(D)} \geq \displaystyle{\gbin{k-1}{i}_q}$ for all $D \in \displaystyle{\gbin{V - A}{i}}$, $\FF^\star = \emptyset$, and $\size{\phi_A(B)} = q^{i(k-i)}$ for all $B \in \FF_i \cup \FF_{k-i}$. 
\end{thm}

\begin{proof}
Construct a bipartite graph with vertex set $X \cup Y$, where $X = \gbin{V - A}{i}$ and $Y = \gbin{V - A}{k-i}$. As for edges, $D \in X$ is adjacent to $E \in Y$ if $D \subseteq E$. 

The $(k-i)$-dimensional subspaces of $V$ that contain $D$ and are skew to $A$ are in one to one correspondence with $(k-2i)$-dimensional subspaces of the $(n-i)$-dimensional vector space $V/D$ that are skew to the $k$-dimensional subspace $(A+D)/D$. By Lemma \ref{lem:prelimsub}(\ref{lem:subspacecount}), their number---which is the degree of the vertex $D \in X$---is $q^{(k-2i)k} \gbin{n-k-i}{k-2i}_q$. Likewise, the degree of a vertex $E \in Y$---that is, the number of $i$ dimensional subspaces of an $k-i$ dimensional vector space---is $\displaystyle{\gbin{k-i}{i}_q = \gbin{k-i}{k-2i}_q}$. As a result, for $S \subseteq X$, $\displaystyle{\size{N_Y(S)} \geq \frac{\size{S} q^{(k-2i)k}\gbin{n-k-i}{k-2i}_q}{\gbin{k-i}{k-2i}_q}}$. Now, since $k \leq n/2$, $n-k-i\geq k-i$  we have $\displaystyle{\frac{q^{(k-2i)k}\gbin{n-k-i}{k-2i}_q}{\gbin{k-i}{k-2i}_q} \geq q^{(k-2i)k}}$. Let $m = \begin{cases} q^{k-i} & \textrm{if}\ i < k/2 \\ 1 & \textrm{if}\ i = k/2 \end{cases}$. Since $i \le k/2$ and $k \ge k-i$, we have $\displaystyle{\size{N_Y(S)} \geq m \size{S}}$. By Theorem \ref{thm:normalizedmatching}, there is a one to $m$ matching from elements of $X$ to a subset of $Y$. Arbitrarily match each unmatched element of $Y$ to one of its subsets in $X$. Denote by $\EE_D$, the elements in $Y$ that have been matched with $D \in X$. Now, $\{ \EE_D \mid D \in X\}$ is a partition of $Y$ into $\size{X}$ parts, and, for all $D \in X$, $\size{\EE_D} \ge m$. We can now apply Proposition \ref{prop:phiinverseDED}, and get
\begin{align*}
\sum_{D \in X} \size{\phi_A^{-1}(D)} + \sum_{E \in Y} \size{\phi_A^{-1}(E)} & = \sum_{D \in X} \left(\size{\phi_A^{-1}(D)} + \sum_{E \in \EE_D} \size{\phi_A^{-1}(E)}\right) \\
& \leq  \sum_{D \in X} \left(\gbin{k-1}{i}_q + \size{\EE_D} \gbin{k-1}{i-1}_q\right) \\
& = \gbin{k-1}{i}_q \size{X} + \gbin{k-1}{i-1}_q \size{Y} \\
& = q^{ki} \gbin{n-k}{i}_q\gbin{k-1}{i}_q + q^{k(k-i)} \gbin{n-k}{k-i}_q \gbin{k-1}{i-1}_q.
\end{align*}
To complete the proof, we will use a double counting argument to get the desired bound on $\size{\FF_i} + \size{\FF_{k-i}}$.

Let $\Ss = \{(B, D) \mid B \in \FF_i, D \in \phi_A(B), \dim(D) = k-i\} \cup \{(B, D) \mid B \in \FF_{k-i}, D \in \phi_A(B), \dim(D) = i\}$. As a technical note, if $i = k/2$, we allow $\Ss$ to be a multi-set and count each element twice. Now, for each $B \in \FF_i \cup \FF_{k-i}$, there are at least $q^{i(k-i)}$ suitable $D$'s (Lemma \ref{lem:prelim}\eqref{lem:prelime}\&\eqref{lem:prelimf}), and so $\size{\Ss} \geq (\size{\FF_i} + \size{\FF_{k-i}})q^{i(k-i)}$. On the other hand, the $D$'s are in $X \cup Y$, and, for each $D \in X \cup Y$, the number of $(B, D) \in \Ss$ is $\size{\phi^{-1}_A(D)}$. Hence, $\size{\Ss} = \sum_{D \in X} \size{\phi_A^{-1}(D)} + \sum_{E \in Y} \size{\phi_A^{-1}(E)} \leq q^{ki} \gbin{n-k}{i}_q\gbin{k-1}{i}_q + q^{k(k-i)} \gbin{n-k}{k-i}_q \gbin{k-1}{i-1}_q$, by the earlier calculation. Putting the two inequalities together, we get
$$\size{\FF_i} + \size{\FF_{k-i}} \leq \frac{\size{\Ss}}{q^{i(k-i)}} \leq q^{i^2} \gbin{n-k}{i}_q\gbin{k-1}{i}_q + q^{(k-i)^2} \gbin{n-k}{k-i}_q \gbin{k-1}{i-1}_q,$$
as desired. 

If Inequality \eqref{eq:sizeFiFk-i} is an equality, then we have to also have equality in Inequality \eqref{eq:phiinvDE} of Proposition \ref{prop:phiinverseDED}, and so $\size{\phi_A^{-1}(D)} \geq \displaystyle{\gbin{k-1}{i}_q}$ for all $D \in \displaystyle{\gbin{V - A}{i}}$. Furthermore, to have equality, in the count for $\size{S}$, for each $B \in \FF_i \cup \FF_{k-i}$, we need exactly $q^i(k-i)$ suitable $D$'s. For $B \not\in \FF^\star$, this means that $\size{\phi_A(B)} = q^{i(k-i)}$, and Lemma \ref{lem:prelim}\eqref{lem:prelimf} implies that none of the elements of $\FF$ are in $\FF^\star$. The proof is now complete.  \qedhere
\end{proof}

We are now ready to prove our main result, Theorem A of the introduction.

\begin{thm}\label{thm:Main}
Let $n$ be a positive integer, $q$ a prime power,  $V$ a vector space of dimension $n$ over $\F_q$, $2 \leq k \leq n/2$, and $\FF \subseteq \displaystyle{\gbin{V}{k}}$  a covering-triple-free family. Then 
\begin{equation}
\size{\FF} \leq \gbin{n-1}{k-1}_q. \label{eq:finalresult}
\end{equation}
Moreover, for $2 \leq k < n/2$, if $\size{\FF} = \displaystyle{\gbin{n-1}{k-1}_q}$, then $\FF$ is a star.
\end{thm}

\begin{proof} If possible, choose $A, B^\star \in \FF$ with $A \cap B^\star = \{\zo_V\}$, otherwise, choose $A \in \FF$ arbitrarily. As we have been doing all along, let $\FF_i = \FF_{i, A}$ be defined as in Definition \ref{def:iIphi}. Since $\displaystyle{\FF\backslash \{A\} = \cup_{i = 1}^{k-1} \FF_i}$, we have $\displaystyle{\size{\FF} \leq 1 + \sum_{i=1}^{k-1} \size{\FF_i}}$. For each $1 \leq i \leq \floor{(k-1)/2}$, we pair $\FF_i$ with $\FF_{k-i}$, in order to use Theorem \ref{thm:FFiFFkmi} to bound $\size{\FF_i} + \size{\FF_{k-i}}$. Note that in the case when $k$ is even, $\FF_{k/2}$ will be unmatched, but Theorem \ref{thm:FFiFFkmi} will give us a bound on $|\FF_{k/2}|$ which is still sufficient. To treat the odd and even cases simultaneously, define $\delta$ to be equal to one if $k$ is even and zero otherwise. Using Theorem \ref{thm:FFiFFkmi}, a straightforward change of index, and Lemma \ref{lem:starcount}, we now have
\begin{align*}
\size{\FF} & \leq 1+\sum_{i=1}^{k-1} \size{\FF_i}   =1 + \delta \size{\FF_{k/2}} +  \sum_{i = 1}^{\floor{(k-1)/2}} \left(\size{\FF_i} + \size{\FF_{k-i}}\right) \\
& \leq 1 + \delta q^{(k/2)^2}\gbin{n-k}{k/2}_q \gbin{k-1}{k/2-1}_q  \\ & \qquad+ \sum_{i = 1}^{\floor{(k-1)/2}}\left(q^{(k-i)^2}\gbin{n-k}{k-i}_q \gbin{k-1}{i-1}_q + q^{i^2} \gbin{n-k}{i}_q \gbin{k-1}{i}_q\right) \\
& = \sum_{i=1}^{k} q^{(k-i)^2}\gbin{n-k}{k-i}_q \gbin{k-1}{i-1}_q = \gbin{n-1}{k-1}_q.
\end{align*}
It remains to account for equality in Inequality \eqref{eq:finalresult}, and so assume that $2 \leq k < n/2$ and $\size{\FF} = \displaystyle{\gbin{n-1}{k-1}_q}$.  First note that if it was \emph{not} possible to choose $A, B^\star \in \FF$ with $A \cap B^\star = \{\zo_V\}$, then the family $\FF$ would be an intersecting family, and by Theorem \ref{thm:vsEKR}, $\FF$ is a maximum-sized star. So assume that $A, B^\star \in \FF$ with $A \cap B^\star = \{\zo_V\}$. We will complete the proof by showing that this is, in fact, not possible. In establishing Inequality \eqref{eq:finalresult} above, we used Inequality \eqref{eq:sizeFiFk-i} of Theorem \ref{thm:FFiFFkmi}. So in the case of equality in the former, the latter must also be an equality, and so, by Theorem \ref{thm:FFiFFkmi}, this means that $\FF^\star = \emptyset$, $\size{\phi_A(B)} = q^{i(k-i)}$ for all $B \in \FF_i$, and $\size{\phi_A^{-1}(D)} \geq \gbin{k-1}{1}_q$ for all $D \in \gbin{V - A}{1}$. 

We handle first the case $k=2$. Since $B^* \notin \FF^*$, there must be $C^* \in \FF$ such that $D := B^* \cap C^* \neq \{\zo_V\}$. By the equality condition in Lemma \ref{lem:prelim}\eqref{lem:prelime}, we must have $C' \cap B^* = D$ for all $C' \in \II_A(B^*)$, and thus $D \notin \phi_A(B^*)$. However, since $|\phi_A^{-1}(D)| \ge 1$, there must be $B \neq B^*$ with $D \in \phi_A(B)$. Again, since $B \notin \FF^*$ by Lemma \ref{lem:prelim}\eqref{lem:prelimc}, this implies $B = D \oplus (B \cap C)$ for some $C \in \II_A(B).$ Since $D \subset B \cap B^*$ and $\dim(B) = \dim(B^*) = 2$, we have $D = B \cap B^*$. Thus $B = (B \cap B^*) \oplus (B \cap C)$, which is a contradiction.

Now, suppose $k \ge 3$ and let $D \in \gbin{B^*}{1}$. In this case, $|\phi^{-1}_A(D)| \ge \gbin{k-1}{1}_q = q^{k-2}+\cdots+q+1\geq q+1 \geq 3$. Hence, if we let $D \in \gbin{B^*}{1}$, there exists $B \in \phi_A^{-1}(D)$ with $B \neq B^\star$ (regardless of whether $B^\star$ itself is or is not in $\phi_A^{-1}(D)$).

We claim that $B \cap A = \{\zo_V\}$. If this were not the case, then $\II_A(B) = \{A\}$, and by Lemma \ref{lem:prelim}\eqref{lem:prelimc}, $B = (B \cap A) \oplus D =(B \cap A) \oplus (B \cap B^\star)$. This would mean $B$, $A$, and $B^\star$ form a covering triple, and that would be a contradiction. Since $B \in \phi_A^{-1}(D)$ and $B \cap A = \{\zo_V\}$, there exists $C \in \II_A(B)$ with $B = D \oplus (B \cap C)$. Since $D$ is one-dimensional, $B \cap C$ is $(k-1)$-dimensional, and $i_A(B) = k-1$ (see Figure \ref{fig:phiAB} and Lemma \ref{lem:prelim}). Now $B \cap B^\star \neq \{\zo_V\}$ since $D$ is in that intersection. Hence, by Lemma \ref{lem:prelim}\eqref{lem:prelimb}, $\dim(B \cap B^\star) \geq i_A(B) = k-1$. But $B$ and $B^\star$ are distinct $k$ dimensional subspaces, and so $B \cap B^\star$ is a $(k-1)$-dimensional subspace. This means $B^\star \in \II_A(B)$. By Lemma \ref{lem:prelim}\eqref{lem:prelime}, since $\size{\phi_A(B)} = q^{i(k-i)}$, and since $C, B^\star \in \II_A(B)$, we must have $B^\star \cap B = C \cap B$. But this is a contradiction since $D \subseteq B^\star \cap B$, while $D \cap C = \{\zo_V\}$. The contradiction completes the proof.
\end{proof}

\section{Concluding remarks and further directions}
The natural next step is to consider the question of $d$-clusters for vector spaces. In light of this, we make the following definition.

\begin{definition}
    Let $V$ be an $n$-dimensional vector space over $\mathbb{F}_q$, and suppose $A_1,\dots,A_d \in \gbin{V}{k}$. If $\dim(A_1+\dots+A_d)\le 2k$ and $A_1 \cap \dots \cap A_d = \{\zo_V\}$, then we say that $A_1,\dots,A_d$ form a \emph{$d$-cluster}. A family $\FF \subset \gbin{V}{k}$ containing no $d$-cluster is called \emph{$d$-cluster-free}.
\end{definition}

We then conjecture the following

\begin{conj}
Let $V$ be an $n$-dimensional vector space over $\FF_q$, and let $k \ge d \ge 3$ be integers satisfying $n \ge \frac{dk}{d-1}$. Then, if $\FF \subset \gbin{V}{k}$ is $d$-cluster-free, it must satisfy $$|\FF| \le \gbin{n-1}{k-1}_q$$ with equality only if $\FF$ is a maximum-sized star.
\end{conj}
We note that this range of $n,d,k$ is the only for which the question is meaningful, since if $n < \frac{dk}{d-1}$, then any collection of $d$ vector spaces of dimension $k$ will have non-trivial intersection.

Our Theorem \ref{thm:Main} proves this conjecture for $d=3$ in the slightly limited range $n \ge 2k$ (although we did not prove equality in the case $n = 2k$). This is not entirely surprising, since even in the case of sets the cases $n \ge 2k$ and $n < 2k$ often have to be handled separately. In addition to tackling the case of $n < 2k$, it would be interesting to see if the tools used to attack the problem in sets for $d \ge 4$ (see e.g. Mubayi \cite{Mubayi:07}, Mubayi and Ramadurai \cite{MubayiRam:09} and Currier \cite{Currier:21}) could be applied to the context of vector spaces.

\newcommand{\ugst}{\relax}



\providecommand{\bysame}{\leavevmode\hbox to3em{\hrulefill}\thinspace}
\providecommand{\MR}{\relax\ifhmode\unskip\space\fi MR }
\providecommand{\MRhref}[2]{%
  \href{http://www.ams.org/mathscinet-getitem?mr=#1}{#2}
}
\providecommand{\href}[2]{#2}

\end{document}